\documentclass[12pt]{article}
\usepackage[letterpaper, hmargin=1.55cm, vmargin={1.8cm, 1.8cm}]{geometry}
\usepackage{myDefs}

\begin{document}
\title{Harmonic measure of the outer boundary of colander sets}
\author{Adi Gl\"ucksam\thanks{Supported in part by Schmidt Futures program.}}
\maketitle
\begin{abstract}
We present two companion results: Phragm\'en-Lindel\"of type tight bounds on the minimal possible growth of subharmonic functions with recurrent zero set, and tight bounds on the maximal possible decay of the 
harmonic measure of the outer boundary of colander sets.
\end{abstract}
\section{Introduction}
\begin{minipage}[c]{.49\textwidth}
Given a non-decreasing positive concave function $R(r)=o(r),\;r\rightarrow\infty$, and a non-increasing positive function $\eps:\R_+\rightarrow\bb{0,1}$ we say that a closed set $E\subset\R^d$ is {\it $(\eps,R)$-recurrent} if for every $x\in\R^d$:
\begin{eqnarray}\label{eq:recurrency}
\kib d\bb{B(x,R(\abs x))\cap E}>R\bb{\abs x}\eps\bb{\abs x},
\end{eqnarray}
where $\abs\cdot$ denotes the euclidean norm in $\R^d$, and $\kib d(A)$ denotes the capacity of the set $A$ (the Newtonain capacity for $d\ge 3$ and the logarithmic capacity for $d=2$). For the precise definition of capacity that we use in this paper, we refer the reader to chapter 5 in \cite{Heywood}. 
\end{minipage}\hfill
\begin{minipage}[c]{.47\textwidth}
\centering
\includegraphics[width=0.95\textwidth]{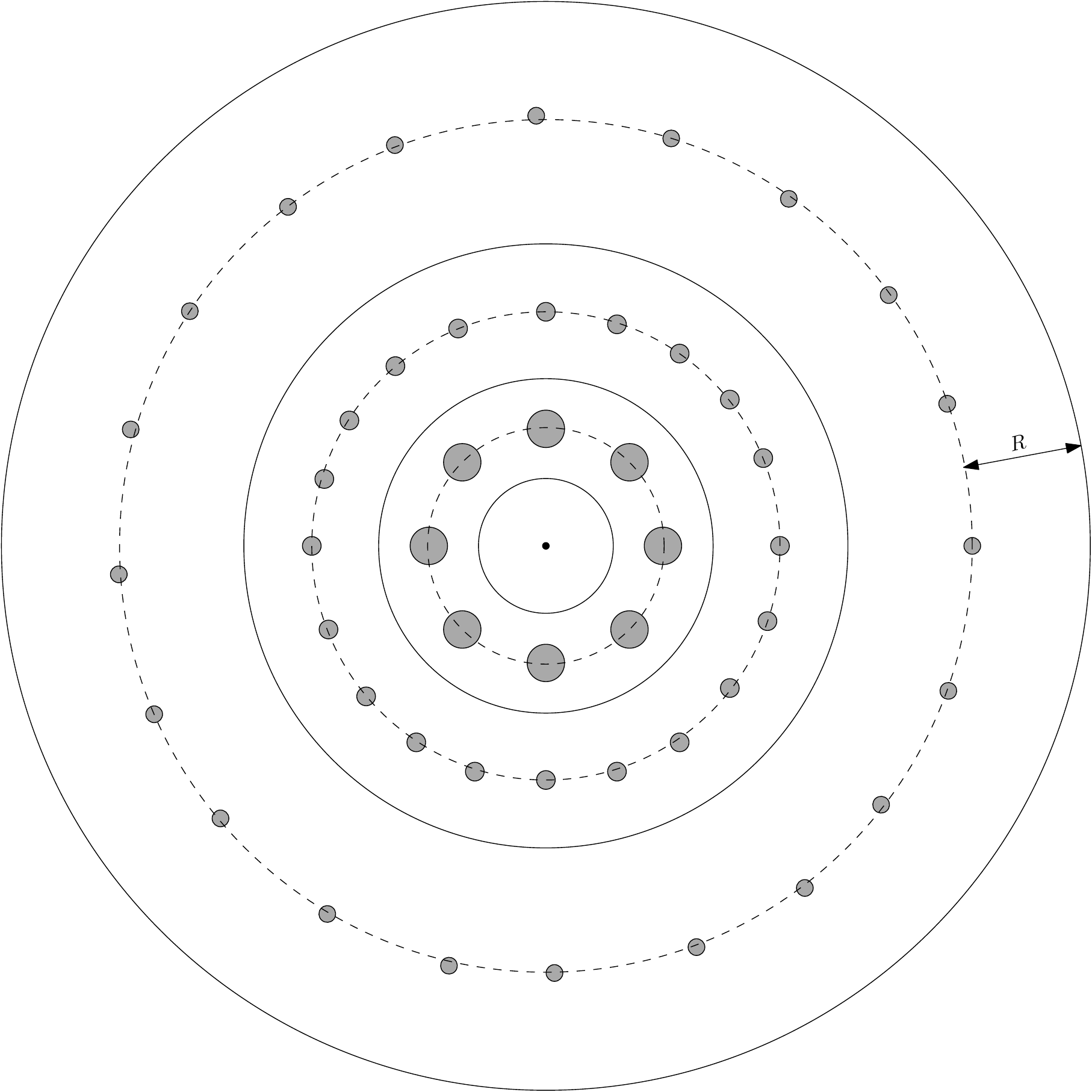}
\captionof{figure}{The function $R(t)$ determines the distribution of the set in space, the function $\eps(t)$ determines its relative size.}
\end{minipage}~\\~\\
We compare the relative capacity of the set $E$ inside the ball $B(x,R(\abs x))$ with the capacity of a ball of radius $\eps(\abs x)$, which is its radius. Condition (\ref{eq:recurrency}) can be rewritten as
$$
\frac{\kib d\bb{B(x,R(\abs x))\cap E}}{\kib d(B(x,R(\abs x)))}>\kib d(B(x,\eps(\abs x))).
$$
While the function $R$ determines the distribution of our set in space, the function $\eps$ determines the relative size of it.

We say a set $\Omega$ is {\it $(\eps,R)$-colander} if there exists $\rho>R(0)$ and an $(\eps,R)$-recurrent set $E$ so that $\Omega=B_\rho\setminus E$, where $B_\rho:=\bset{\abs x<\rho}$.\

In this paper we are considering two related objects in potential theory: the decay of the harmonic measure of the outer boundary of colander sets, and 
Phragm\'en-Lindel\"of type theorems for subharmonic functions with a recurrent zero set.

The question of the optimal growth of subharmonic functions with $(\eps,R)$-recurrent zero set, originated in a joint work with Buhovsky, Logunov and Sodin from 2017 (see \cite{Us2017}). The harmonic measure counterpart is a natural sequel.

We give accurate asymptotic estimates for the harmonic measure $\omega(0,\partial B_\rho;\Omega)$ for colander sets $\Omega$, showing that this harmonic measure decays like
$$
\exp\bb{-c\integrate 1{\rho}{\frac1{R(t)\sqrt{-\ker d(\eps(t))}}}t}\;\;, \text{ for }\;\; \ker d(t):=	\begin{cases}	
															\log (t)&, d=2\\
															\frac{-1}{t^{d-2}}&,d\ge 3
														\end{cases},
$$
while also giving precise bounds on the growth of subharmonic functions, whose zero set is recurrent showing that
$$
\log M_u(r):=\log\bb{\underset{\abs x\le r}\max\; u(x)}\sim \integrate 1{\rho}{\frac1{R(t)\sqrt{-\ker d(\eps(t))}}}t,
$$
where the notation $A\sim B$ indicates that 
$$
A\lesssim B \text{ and }B\lesssim A,
$$
and $A\lesssim B$ if there exists a constant $\alpha>0$ so that $A\le \alpha\cdot B$.

Other than being an interesting object on their own, harmonic measures arise in the context of Brownian motion, and they are useful tools in estimating the optimal growth of subharmonic functions that are bounded on sets with analytic boundary. For those reasons and more, harmonic measures have been the center of interest for many people. The matter of how small must a well distributed set be in order to be ignored has been investigated in many different contexts. One aspect is whether these sets are avoidable or not, i.e is the harmonic measure supported on this set, and assigns zero to the boundary of the disk. The question of when a set is avoidable has been investigated by many mathematicians: Akeroyd \cite{Akeroyd2002}, Carrol and Ortega-Cerd\`{a} \cite{Orthega2007}, O'Donovan \cite{Odonovan2010}, Gardiner and Ghergu \cite{Gardiner2010}, Pres \cite{Pres2012}, Hansen and Netuka \cite{Hansen2013} and more... An overturn of this question would be to ask when are such sets so small that the remaining harmonic measure, the one restricted to the boundary of the disk, is comparable with Lebesgue's measure. Questions such as these have been answered by Volberg  \cite{Volberg1991}, Ess\'en \cite{Essen1993} and by Aikawa and Lundh \cite{Aikawa1996}. Other aspects are the Hausdorff dimension of the support of harmonic measures, and their density. These aspects have been investigated throughly by \O ksendal \cite{Oksendal1981}, Bourgain \cite{Bourgain1987},  Jones and Wolf \cite{JoneWolf1988},  Jones and Makarov \cite{JoneMakarov1995}, Ortega-Cerd\`{a} and Seip \cite{Orthega2004}, and more... 
\subsection{Results}
Before stating the results, we remind the reader the definition of the capacity kernel used in \cite{Heywood}:
$$
\ker d(t):=	\begin{cases}	
		\log (t)&, d=2\\
		\frac{-1}{t^{d-2}}&,d\ge 3
		\end{cases}.
$$
We are now ready to present our results: let
$$
\varphi(t)=\varphi_{\eps,R}(t):=\frac1{R(t)\sqrt{-\ker d(\eps(t))}}.
$$
\begin{thm}\label{thm:function}
\begin{enumerate}[label=(\Alph*)]
\item\label{thm:functionA} Assume that $\limitsup t\infty\frac1{t\cdot\varphi_{\eps,R}(t)}<1$, and let $u$ be a subharmonic function in $\R^d$ so that its zero set is $(\eps,R)$-recurrent, and there exists $x_0\in\R^d$ so that $u(x_0)\ge 1$. Then
$$
\limitinf \rho\infty\frac{\log M_u(\rho)}{\integrate 1{\rho}{\varphi_{\eps,R}(t)}t}>0.
$$
\item\label{thm:functionB} If $\frac d{dt}\bb{\frac1{\varphi_{\eps,R}(t)}}$ is bounded, then there exists a subharmonic function, $u$ in $\R^d$ whose zero set is $(\eps,R)$-recurrent, while $u(0)\ge1$ and
$$
\limitsup \rho\infty\frac{\log\bb{M_u(\rho)}}{\integrate 1\rho{\varphi_{\eps,R}(t)}t}<\infty.
$$
\end{enumerate}
\end{thm}
\begin{rmk}
In the case $d=2$ any non-constant subharmonic function is unbounded, and the requirement that there exists $x_0\in\R^d$ so that $u(x_0)\ge 1$ is not needed. Never the less, if $d\ge 3$, then the function $u(x)=\frac{-1}{\abs x^{d-2}}$ is subharmonic, non-constant and $Z_u=\R^d$, making this requirements necessary.
\end{rmk}
\begin{rmk}
Though the conditions imposed in the theorem above seem harsh, they are quite weak. In fact, if the function $t\mapsto\frac1{\varphi(t)}$ is a gauge function (see definition bellow), it will satisfy both conditions.
\begin{defn}
For every $n$ we denote by $\log_{[n]}(t)$ the $n$th iterated logarithm of $t$, i.e
$$
\log_{[n]}(t):=\begin{cases}
                t&, n=0\\
                \log\bb{\log_{[n-1]}(t)}&,n\ge 1
               \end{cases}.
$$
We say a function $f$ is a \underline{gauge function} if it is monotone non-decreasing, $\underset{t\in\R}\sup\; \frac{f(t)}t<\infty$, and
$$
f(t)=A\prodit n 0 \infty \log^{\alpha_n}_{[n]}(t),
$$
where $A>0$, $\alpha_0\in[0,1],\;\alpha_n\in\R,\; n\ge 1$, and $\#\bset{n,\;\alpha_n\neq 0}<\infty.$
\end{defn}
\end{rmk}
As mentioned earlier, there are two objects at play here: subharmonic functions, and harmonic measures. The relationship between the two is more conceptional than formal, though there is also some formal connection, as we will soon see. The following theorem describes tight bounds for the decay of the harmonic measure $\omega(0,\partial B_\rho;\Omega)$ for $(\eps,R)$-colander sets $\Omega$.
\begin{thm}\label{thm:measure}
\begin{enumerate}[label=(\Alph*)]
\item\label{thm:measureA} If $\limitsup t\infty\frac1{t\cdot\varphi_{\eps,R}(t)}<1$, then there exist constants $C,c>0$ so that for every  $(\eps,R)$- recurrent set $E$, for every $\rho>1$ the $(\eps,R)$-colander set, $\Omega:=B_\rho\setminus E$ satisfies
$$
\omega\bb{0,\partial B_\rho;\Omega}\le C\exp\bb{-c\integrate 1{\rho}{\varphi_{\eps,R}(t)}t}.
$$
\item\label{thm:measureB} If $\frac d{dt}\bb{\frac1{\varphi_{\eps,R}(t)}}$ is bounded, then there exist constants $c,C>0$, and an $(\eps,R)$- recurrent set $E$, so that for every $\rho>1$ the $(\eps,R)$-colander set, $\Omega:=B_\rho\setminus E$ satisfies
$$
\omega\bb{0,\partial B_\rho;\Omega}\ge C\exp\bb{-c\integrate 1{\rho}{\varphi_{\eps,R}(t)}t}.
$$
\end{enumerate}
\end{thm}
\begin{rmk}{\normalfont
Wiener's criterion for thin sets states that for $\gamma\in(1,\infty)$ and 
$$
E_n:=E\cap\bset{\gamma^{n-1}\le\abs z<\gamma^n},
$$
the set $E$ is thin at infinity if and only if the following series converges
$$
\sumit n 1 \infty\frac{\ker d(\gamma^n)}{\ker d(-\kib d(E_n))}=	\begin{cases}
												\log(\gamma)\sumit n 1 \infty\frac{n}{\log\bb{\frac1{\kib d(E_n)}}}&, d=2\\
												\sumit n 1 \infty \gamma^{n(d-2)}\kib d(E_n)^{d-2}&, d\ge 3
											\end{cases}.
$$
Heuristically, if Theorem \ref{thm:measure} \ref{thm:measureA} holds in general, without the additional condition that $\limitsup t\infty \frac1{t\varphi(t)}<1$, then the set $E$ is thin at infinity if there exists $\delta>0$ so that
$$
\varphi(t)<		\begin{cases}
				\frac1{t\log^{2+\delta}(t)}&, d=2\\
				\frac1{t^{d-2+\delta}}&, d\ge 3
			\end{cases}.
$$
This means there is still a gap between the case where our set is thin and the restriction posed in Theorem \ref{thm:measure} \ref{thm:measureA}, and moreover this gap becomes bigger 
the higher the dimension.}
\end{rmk}
\subsection{An overview of the paper: Methods and Tools}
Theorems \ref{thm:function} and \ref{thm:measure}, though conceptually equivalent, do not formally imply one another. Nevertheless, there is no need to prove both theorems:
\begin{lem}
\begin{enumerate}[label=(\roman*)]
\item If Theorem \ref{thm:measure} \ref{thm:measureA} holds then Theorem \ref{thm:function} \ref{thm:functionA} holds.
\item If Theorem \ref{thm:function} \ref{thm:functionB} holds then Theorem \ref{thm:measure} \ref{thm:measureB} holds.
\end{enumerate}
\end{lem}
\begin{proof}
To prove both parts of the theorem, we will use the fact that if $u$ is a subharmonic function with zero set $Z_u:=\bset{u\le 0}$, then by definition of harmonic measure, for every $\rho>1$
\[
u(0)=\underset{\partial\bb{B_\rho\setminus Z_u}}\int u(x)d\omega\bb{0,x;B_\rho\setminus Z_u}\le M_u(\rho)\cdot \omega\bb{0,\partial B_\rho;B_\rho\setminus Z_u}.\label{eq:heart}\tag{$\heartsuit$}
\]
\paragraph{The proof of {\it (i)}:} Let $u$ be a subharmonic function in $\R^d$ so that its zero set is $(\eps,R)$-recurrent, and let $x_0$ be so that $u(x_0)\ge 1$. Define the function $v(x)=u(x+x_0)$, then $v$ is a subharmonic function so that $v(0)\ge 1$, and its zero set, $Z_v$, is $(\tilde\eps,\tilde R)$-recurrent for $\tilde\eps(t)=\eps(t+\abs{x_0}),\; \tilde R(t)=R(t+\abs{x_0})$. The set $B_\rho\setminus Z_v$ is an $(\tilde\eps,\tilde R)$-colander set, and by Theorem \ref{thm:measure} \ref{thm:measureA}, there exist $0<c<C<\infty$ satisfying
$$
\omega\bb{0,\partial B_\rho;B_\rho\setminus Z_v}\le C\exp\bb{-c\integrate 1{\rho}{\varphi_{\tilde\eps,\tilde R}(t)}t},\;\;\forall\rho>1.
$$
Combining this with (\ref{eq:heart}) we get
\begin{eqnarray*}
	v(0)&\le&M_v(\rho)\cdot \omega\bb{0,\partial B_\rho;B_\rho\setminus Z_v}\le CM_v(\rho)\cdot \exp\bb{-c\integrate 1{\rho}{\varphi_{\tilde\eps,\tilde R}(t)}t},
\end{eqnarray*}
implying that 
\begin{eqnarray*}
M_v(\rho)&\ge&\frac{v(0)}C\exp\bb{c\integrate 1{\rho}{\varphi_{\tilde\eps,\tilde R}(t)}t}\ge\frac{1}C\exp\bb{c\integrate 1{\rho}{\varphi_{\eps, R}(t+\abs{x_0})}t}\\
&=&\frac{1}C\exp\bb{c\integrate {1+\abs{x_0}}{\rho+\abs{x_0}}{\varphi_{\eps, R}(t)}t}\ge\frac{\exp\bb{c\integrate 1{\rho}{\varphi_{\eps, R}(t)}t}}{C\exp\bb{c\integrate 1{1+\abs{x_0}}{\varphi_{\eps, R}(t)}t}}.
\end{eqnarray*}
We see that
$$
\limitinf \rho\infty\frac{\log M_v(\rho)}{\integrate 1{\rho}{\varphi(t)}t}=\limitinf \rho\infty\frac{\log M_u(\rho)}{\integrate 1{\rho}{\varphi(t)}t}>0.
$$
\paragraph{The proof of {\it (ii)}:} Let $u$ be the non-constant subharmonic function constructed in Theorem \ref{thm:function} \ref{thm:functionB}, whose zero set, $Z_u$, is $(\eps,R)$-recurrent, while $u(0)\ge 1$ and
$$
M_u(\rho)\le C\cdot\exp\bb{c\integrate 1\rho{\varphi(t)}t}
$$
for some constants $c,C>0$. Then, following (\ref{eq:heart}),
$$
1\le u(0)\le M_u(\rho)\cdot \omega\bb{0,\partial B_\rho;B_\rho\setminus Z_u}\le C\cdot\exp\bb{c\integrate 1\rho{\varphi(t)}t}\cdot \omega\bb{0,\partial B_\rho;B_\rho\setminus Z_u}
$$
concluding the proof.
\end{proof}
We will indeed prove Theorem \ref{thm:measure} \ref{thm:measureA}, and  Theorem \ref{thm:function} \ref{thm:functionB}. As a corollary to the lemma above, we obtain the proofs for both Theorem 
\ref{thm:function} and Theorem \ref{thm:measure}.
\subsubsection{The sketch of proof of Theorem \ref{thm:measure} \ref{thm:measureA}- Section\ref{sec:uprbnd}.}
The first key element of the proof is observing that to understand the asymptotic decay of the harmonic measure it is enough to understand its decay on a sequence of annuli, which we will call `layers'. The second key component in the proof is that the number of layers, one needs to consider in order to obtain a meaningful bound, depends on the function $\eps(t)$. Though the number of layers is not a lot (proportionally to the total number of layers), it is not enough to take a finite number of layers. This implies we do have some dependence between layers, which is quite surprising. To conclude the proof, we bound the harmonic measure of the required layers, by comparing the harmonic measure with a special subharmonic function bounding it from bellow.
\subsubsection{The sketch of proof of Theorem \ref{thm:function} \ref{thm:functionB}- Section \ref{sec:func}}
We conclude the paper with a construction of a subharmonic function whose zero set is $(\eps,R)$-recurrent, proving Theorem \ref{thm:function} \ref{thm:functionB}.\\
The idea of the construction is based on the fact that the function $z\mapsto e^{C\abs z}$ is subharmonic and its Laplacian is very big for $C\gg1$, making the measure $\Delta u\;dm$ much larger than the original function. We can use this extra growth to compensate for adding a large portion of space where $u\le 0$. Now it is only a matter of finding the appropriate candidate to replace the constant $C$, and use some glueing techniques of Poisson integrals to create these large areas where $u\le 0$.
\subsection{Acknowledgements}
The author is lexicographically grateful to Ilia Binder, Alexander Borichev, Steven Britt, John Garnett, and Mikhail Sodin for several very helpful discussions.
\section{A lower bound}\label{sec:uprbnd}
\subsection{Notation and Preliminaries}\label{subsec:not}
Given a concave differentiable monotone increasing function $R:[1,\infty)\rightarrow[1,\infty)$ so that
$$
R(t)=o(t)\,\,\text{ and }\;\;R'(0)=\underset{x\in\R_+}\sup R'(x)<1,
$$
define the sequence
$$
\rho_0:=0,\; \rho_{n+1}=\rho_n+R\bb{\rho_n}.
$$
What can be said about the oscillations of the sequence $\bset{R(\rho_n)}$?
\begin{claim}\label{clm:R_n/R_k}
Let $i(n)=\left\lfloor \sqrt{-\ker d(\eps(\rho_n))}\right\rfloor$ and define $e(\rho_n):=\frac1{\rho_n\varphi(\rho_n)}$. If $e(\rho_n)<1$ then for every $n-i(n)\le k\le n$
$$
\frac{R(\rho_n)}{R(\rho_k)}\le\frac1{1-e(\rho_n)}.
$$
\end{claim}
\begin{proof}
Since $R$ is a monotone non-negative concave function,
\[
\label{eq:a}\tag{a} \frac{R'(t)}{R(t)}\le\frac{R'(t)}{t\cdot R'(t)+R(0)}\le \frac1t.
\]
We deduce that for every $t\ge 1$ and $\eps\in\bb{0,1}$:
\begin{align*}
1\le \frac{R(t)}{R(t\bb{1-\eps})}&\overset{\text{concavity}}\le1+\eps\cdot t\cdot \frac{R'(t\bb{1-\eps})}{R(t\bb{1-\eps})}\overset{\text{by (\ref{eq:a})}}\le1+\frac{\eps}{1-\eps}=\frac1{1-\eps}.\tag{b}\label{eq:b}
\end{align*}
Recall that $\varphi(t)=\frac1{R(t)\sqrt{-\ker d(\eps(t))}}$, therefore
\begin{align*}
\tag{c}\label{eq:c} 1-e(\rho_n)=1-\frac1{\rho_n\varphi(\rho_n)}&=1-\frac{\sqrt{-\ker d(\eps(\rho_n))} R(\rho_n)}{\rho_{n}}\le \frac{\rho_{n-i(n)}\cdot R(\rho_n)}{\rho_{n}}\\
&\le\frac{\rho_n-\sumit j 1 {i(n)} R(\rho_{n-j})}{\rho_{n}}=\frac{\rho_{n-i(n)}}{\rho_n}.
\end{align*}
We conclude that for every $n-i(n)\le k\le n$
\begin{eqnarray*}
\frac{R(\rho_n)}{R(\rho_k)}\le\frac{R(\rho_n)}{R(\rho_{n-i(n)})}=\frac{R(\rho_n)}{R\bb{\rho_n\cdot\frac{\rho_{n-i(n)}}{\rho_n}}}\overset{\text{by (\ref{eq:c})}}\le\frac{R(\rho_n)}{R\bb{\rho_n\bb{1-e(\rho_n)}}}\overset{\text{by (\ref{eq:b})}}\le \frac{1}{1-e(\rho_n)}.
\end{eqnarray*}
\end{proof}
\begin{rmk}
In fact, we only use the condition $\limitsup t\infty \frac1{t\varphi(t)}<1$ to get an upper bound on the quotient $\frac{R(\rho_n)}{R(\rho_{n-i(n)})}$. If this quotient is known to be bounded (for example if $R$ is constant), this condition is not needed at all.
\end{rmk}
\paragraph{Layers of independence}
The strong connection between harmonic measures and Brownian motion gives analysts intuition when it comes to estimating harmonic measures. In probability, the simplest case to deal with is when you have independence between events, allowing you to tightly bound the probability of each event separately to obtain a tight bound on the intersection. Would it not be wonderful if we had independence between different layers\footnote{Here and elsewhere in the paper we use {\it layer} to refer to the intersection between the set $E$ and some annulus.} of our set?\\
The strong Markov property of Brownian motion tells us that even when we do not have independence, there is some kind of relatively weak connection between layers. This was definitely known and used by experts in the field. We formally describe this "folklore" statement in the following Proposition, whose proof can be found in the appendix:
\begin{prop}\label{prop:layers}
For every closed set $V$ and for every sequence of domains $0\in D_1\Subset D_2\Subset\cdots$,
\begin{eqnarray*}
\exp\bb{-\alpha\sumit k 1 n \underset{x\in\partial D_k}\sup\;\omega\bb{x,V;D_{k+1}\setminus V}}&\le&\omega\bb{0,\partial D_{n+1};D_{n+1}\setminus V}\\
&\le&\exp\bb{-\sumit k 1 n \underset{x\in\partial D_k}\inf\;\omega\bb{x,V;D_{k+1}\setminus V}},\nonumber
\end{eqnarray*}
provided that $\alpha>1$ and
$$
\underset k\sup\;\underset{x\in\partial D_k}\sup\;\omega\bb{x,V;D_{k+1}\setminus V}\le1-\frac1\alpha.
$$
\end{prop}
To conclude this subsection we would like to relate the estimates done in Proposition \ref{prop:layers} with the sequence $\bset{\rho_n}$:
\begin{cor}\label{cor:layered_integral} Let $F\subset\R^d$ be a closed set and let $D_1\Subset D_2\cdots$ be a sequence of sets so that
$$
\partial D_k\subset \bset{x\in\R^d,\; \alpha\rho_k\le\abs x\le \beta\rho_k}
$$
for some uniform constants $0<\alpha<\beta<\infty$. Let $g_1,g_2:\R_+\rightarrow\R_+$ be continuous  monotone non-increasing functions satisfying that for every $k$ and every $x\in\partial D_k$:
$$
g_1(\rho_k)\le \omega\bb{x,F;D_{k+1}\setminus F}\le g_2(\rho_k)\le 1-\frac1\alpha,
$$
for some uniform constant $\alpha>1$. Then there exists constants $0<B\le A$ so that for every $n$:
$$
\exp\bb{-A\integrate 0{\rho_n}{\frac{g_2(t)}{R(t)}}t}\le \omega\bb{0,\partial D_{n+1};D_{n+1}\setminus F}\le\exp\bb{-B\integrate 0{\rho_n}{\frac{g_1(t)}{R(t)}}t}.
$$
\end{cor}
\begin{proof}
Following Proposition \ref{prop:layers},
$$
\exp\bb{-\alpha\sumit k 1 n \underset{\xi\in\partial D_k}\sup\;\omega\bb{\xi,F;D_{k+1}\setminus F}}\le\omega\bb{0,\partial D_n,D_n\setminus F}\le \exp\bb{-\sumit k 1 n \underset{\xi\in\partial D_k}\inf\;\omega\bb{\xi,F;D_{k+1}\setminus F}}.
$$
Combined with the assumptions of the lemma, we see that
$$
 \exp\bb{-\alpha\sumit k 1 n g_2(\rho_k)}\le \omega\bb{0,\partial D_{n+1},D_{n+1}\setminus F}\le \exp\bb{-\sumit k 1 n g_1(\rho_k)}.
$$
To get further estimates on the inequality above, we will need an estimate on $\rho\inv$. Even though we do not have a useful formula for $\rho_n$, we can obtain a formula for the asymptotic behavior of $\rho\inv$. Define the function
$$
\Phi:\R_+\rightarrow\R_+,\; \Phi(x):=\integrate 0 x {\frac1{R(t)}}t.
$$
For every $n\in\N$:
\begin{eqnarray*}
\Phi(\rho_n)=\integrate 0{\rho_n}{\frac1{R(t)}}t=\sumit k 1 n\integrate{\rho_{k-1}}{\rho_k}{\frac1{R(t)}}t.
\end{eqnarray*}
Since $R$ is monotone increasing and concave, for every $k$ 
\begin{eqnarray*}
1\ge\frac{R(\rho_{k-1})}{R(\rho_k)}=1-\frac{R(\rho_k)-R(\rho_{k-1})}{R(\rho_k)}&=&1-\frac{R'(\xi_k)\bb{\rho_k-\rho_{k-1}}}{R(\rho_k)}\\
\ge 1-\frac{\underset{x\in\R_+}\sup R'(x)R(\rho_k)}{R(\rho_k)}&=&1-\underset{x\in\R_+}\sup R'(x)=1-R'(0):=c_R>0,
\end{eqnarray*}
and therefore
$$
c_R\le\frac{R(\rho_{k-1})}{R(\rho_k)}=\frac{\rho_k-\rho_{k-1}}{R(\rho_k)}\le\integrate{\rho_{k-1}}{\rho_k}{\frac1{R(t)}}t\le\frac{\rho_k-\rho_{k-1}}{R(\rho_{k-1})}=1.
$$
Overall,
\begin{eqnarray*}
c_R\cdot n\le\Phi(\rho_n)\le n\Rightarrow \Phi\inv(c_R\cdot n)\le\rho_n\le\Phi\inv(n),
\end{eqnarray*}
since $\Phi\inv$ is monotone increasing as well. The function $g_2$ is monotone non-increasing implying that $g_2\circ\Phi\inv$ is monotone non-increasing and therefore
\begin{eqnarray*}
\sumit j 1 n g_2(\rho_k)&\le&\sumit j 1 n g_2(\Phi\inv(c_R\cdot j))\le\sumit j 1 n \integrate {j-1}{j}{g_2(\Phi\inv(c_R\cdot t))}t=\integrate 0 {n} {g_2(\Phi\inv(c_R\cdot t))}t\\
&\overset{\text{c.o.v}}=&\integrate{\Phi\inv(0)}{\Phi\inv(c_R\cdot n)}{\frac{g_2(\tau)}{c_RR(\tau)}}\tau\le\frac1{c_R}\integrate {0}{\rho_n}{\frac{g_2(\tau)}{R(\tau)}}\tau.
\end{eqnarray*}
To conclude the proof, we use similar estimates to prove the upper bound.
\end{proof}
\subsection{The proof}
In this subsection we will prove Theorem \ref{thm:measure}\ref{thm:measureA}: we show that if $\limitsup t\infty \frac1{t\varphi_{\eps,R}(t)}<1$, then there exist $c>0$ so that for every $(\eps,R)$-recurrent set $E$, the $(\eps,R)$-colander set, $\Omega=B_\rho\setminus E$, satisfies:
\[
\omega\bb{0,\partial B_\rho; \Omega}\le\exp\bb{-c\integrate 1 {\rho}{\varphi(t)}t}.\tag{$\star$}\label{eq:upr_bnd}
\]
We will first observe that it is enough to prove (\ref{eq:upr_bnd}) for the sequence $\bset{A\cdot\rho_n}$, where $A\ge 1$ is some constant. For every $\rho>A\rho_1$ there exists $n\in\N$ so that $A\rho_n\le\rho<A\rho_{n+1}$. Using the maximum principle:
\begin{eqnarray*}
\omega\bb{0,\partial B_\rho; B_\rho\setminus E}&\le&\omega\bb{0,\partial B_{\rho_n}; B_{\rho_n}\setminus E}\overset{\text{by }(\star)}\le \exp\bb{-c\integrate 1 {A\rho_n}{\varphi(t)}t}\\
&\le&\exp\bb{-\frac c2\integrate 1 {A\rho_{n+1}}{\varphi(t)}t}\le  \exp\bb{-\frac c2\integrate 1 {\rho}{\varphi(t)}t},
\end{eqnarray*}
as $\varphi$ is a positive, monotone, non-increasing function, and $A\rho_{n+1}=A\bb{\rho_n+R(\rho_n)}<2A\rho_n$.

To prove (\ref{eq:upr_bnd}) for $\rho=A\rho_n$, we will use Corollary \ref{cor:layered_integral} with the set $E$ and the domains $D_n=B(0,A\rho_n)$. The constant $A$ will depends on dimension alone, and will be chosen later. Following the corollary, it is enough to show that there exists a constant $c$ so that for every $(\eps,R)$-recurrent set $E$ and for every $n\in\N$
$$
\underset{\xi\in\partial D_{n-1}}\inf\;\omega\bb{\xi,E;D_n\setminus E}\ge \frac c{\sqrt{-\ker d(\eps(\rho_n))}}.
$$
We will bound the harmonic measure from bellow at the point $A\rho_{n-1}$. Because the condition on the set $E$ is invariant with respect to rotations, this simplifies the proof on one hand, 
while one can apply the same argument to every element in $\partial D_{n-1}$, by rotating the set and using the fact that harmonic measure is invariant under rotations.

It seems that to bound the harmonic measure at the point $A\rho_{n-1}$, it will not be enough to look at a finite neighbourhood. Never the less, we will only be interested in a relatively small neighborhood of $A\rho_{n-1}$, and the function we construct will be custom-made for this point. However, the construction is the same for any point and the constants are uniform, so the same proof can be applied to any part of $\partial D_{n-1}$ giving us a uniform lower bound.

As we are only interested in the asymptotic decay of the harmonic measure up to multiplication by constants, and the functions $\eps(\cdot),\;\R(\cdot)$ do not change by much in a small enough neighborhood of the point $A\rho_{n-1}$ (see Claim \ref{clm:R_n/R_k}), we may work with $R(\rho_n)$ and $\eps(\rho_n)$  instead of $R(t),\;\eps(t)$.

We begin by partitioning $\R^d$ into half-open half-closed cubes of edge-length $4R(\rho_n)$. For every $I=(i_1,\cdots,i_d)\in\Z^d$ we let $\lambda$ denote the centre of the cube $C_\lambda$:
$$
C_\lambda=\prodit j 1 d\left[4i_j\cdot R(\rho_n),4\bb{i_j+1}R(\rho_n)\right)=\lambda+\left[-2R(\rho_n),2R(\rho_n)\right)^d.
$$
For every $\lambda$, define $E_\lambda=E\cap \overline{B(\lambda,R(\abs{\lambda}))}$, and the sets
$$
U_{m_0}:=B\bb{0,A\rho_{n-1}-\frac{m_0R(\rho_n)}{10}}\cap B\bb{A\rho_n,m_0\cdot R(\rho_n)},\;\; \Gamma=\bset{\lambda,\;C_\lambda\subset U_{m_0}},
$$
for some $m_0\ll n$ that will be chosen later (see Figure \ref{fig:ell_j} bellow). Note that $\bset{E_\lambda}_{\lambda\in\Gamma}$ is a collection of closed pairwise disjoint sets. Moreover, if $i(n)$ is as defined in Claim \ref{clm:R_n/R_k}, and $m_0\le i(n)$, then for every $\lambda\in\Gamma$ 
$$
\mathcal C_d(E_\lambda)\ge \eps(\abs{\lambda})\cdot R(\abs{\lambda})\ge \eps(\rho_n)R(\rho_{n-m_0})\gtrsim\eps(\rho_n)R(\rho_n)
$$
according to the definition of $(\eps,R)$-recurrent set combined with Claim \ref{clm:R_n/R_k}. Define the harmonic function $u_\Gamma:\R^d\setminus\underset{\lambda\in\Gamma}\bigcup E_\lambda\rightarrow\R$ by
$$
u_{\Gamma}(x)=R(\rho_n)^{d-2}\underset{\lambda\in\Gamma}\sum -p_{\mu_\lambda}(x),
$$
where $p_{\mu_\lambda}$ is the potential of the equilibrium measure of the set $E_\lambda$ with kernel
$$
\ker d(t):=	\begin{cases}	
		\log (t)&, d=2\\
		\frac{-1}{t^{d-2}}&,d\ge 3
		\end{cases}.
$$
\subsubsection{Defining the function}
Let $e(\theta)\in S^{d-1}$ be so that $\underset{x\in\partial B(0,A\rho_n)}\sup\;u_{\Gamma}(x)=u_{\Gamma}(A\rho_ne(\theta))$, and define set
$$
\Lambda=e(-\theta)\Gamma=\bset{e(-\theta)\lambda,\;\lambda\in\Gamma}.
$$
Define the harmonic function $u:\R^d\setminus\underset{\lambda\in\Lambda}\bigcup E_\lambda\rightarrow\R$ by
$$
u(x):=u_{\Lambda}(x)=R(\rho_n)^{d-2}\underset{\lambda\in\Lambda}\sum -p_{\mu_\lambda}(x).
$$
We will often use the following notation- for every $x\in \R^d$ we let
$$
\ell_j(x):=\#\bset{\lambda\in\Lambda, \abs{x-\lambda}\in[j\cdot d\cdot R(\rho_n),(j+1)\cdot d\cdot R(\rho_n))}.
$$
Note that if $\left\lceil \frac{dist(x,\Lambda)}{R(\rho_n)}\right\rceil=\tau$, then there exists a constant, $\kappa_d$, which depends on the dimension alone so that
$$
\frac1{\kappa_d}\cdot j^{d-1}\le \ell_{j+\tau}(x)\le{\kappa_d}\cdot j^{d-1},\;\text{ for }\; \frac{m_0}{4d}\le j\le \frac{3m_0}{4d}.
$$
In fact, the right hand inequality holds for all $1\le j\le 2m_0$.
\begin{center}
\includegraphics[width=0.8\textwidth]{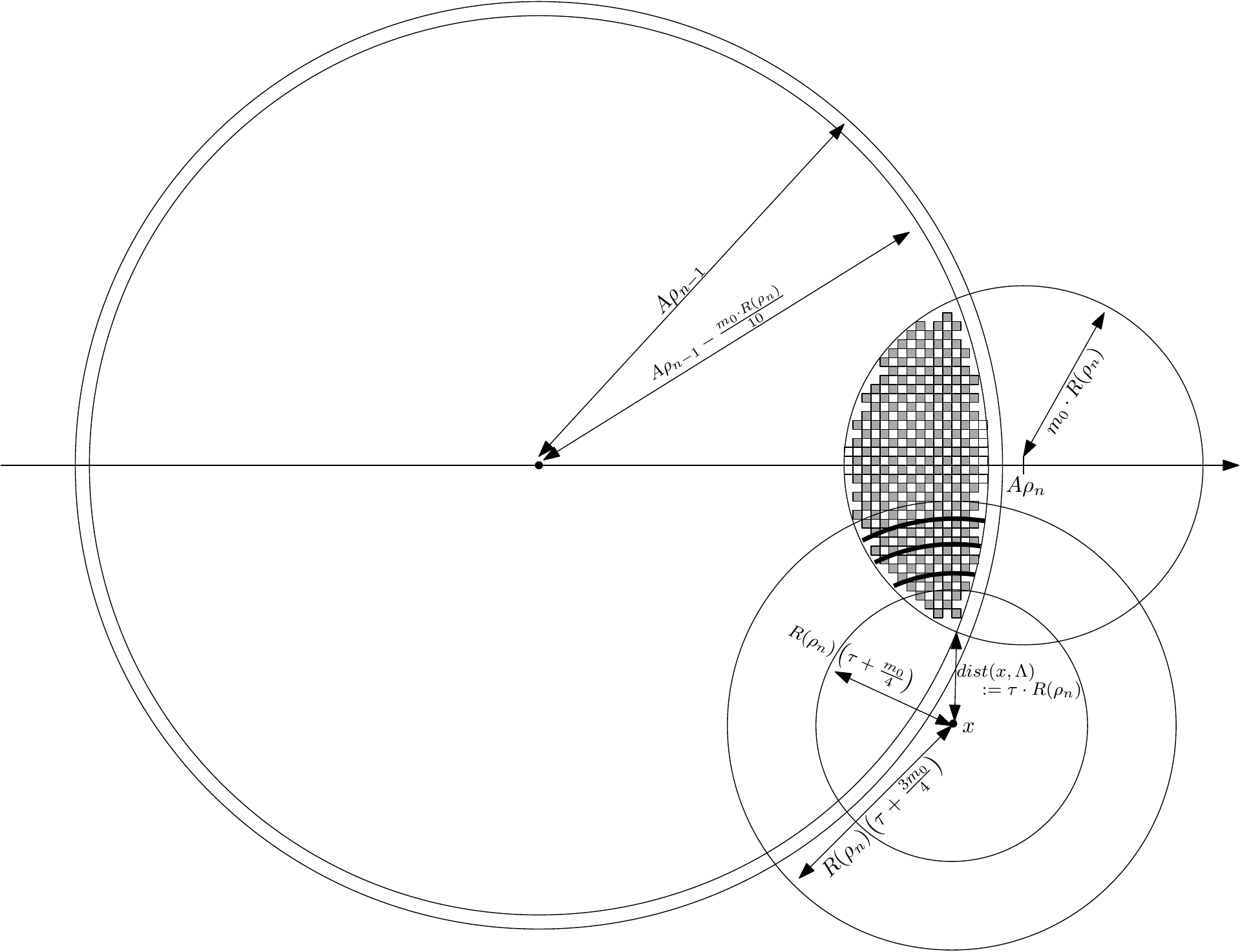}
\captionof{figure}{The grey checkered set is $U_{m_0}$.\\$\ell_j(x)$ counts every cube intersecting the thick black strips, which represent some $\frac{m_0}{4d}\le j\le \frac{3m_0}{4d}$.}
\label{fig:ell_j}
\end{center}
\begin{claim}\label{clm:approx}
Then there exists a constant, $\alpha_d$, which depends on the dimension alone, so that
$$
\abs{u(x)-R(\rho_n)^{d-2}\underset{\lambda\in\Lambda}\sum \bb{-\ker d\bb{\abs{x-\lambda}}\indic{\R^d\setminus C_\lambda}(x)-p_{\mu_\lambda}(x)\cdot\indic{C_\lambda}(x)}}\le\alpha_d\cdot m_0 \;\text{ for every } \;x\in\R^d.
$$
\end{claim}
\begin{proof}
Fix $x\in\R^d$ and let $\lambda\in\Lambda$ be so that $x\nin C_\lambda$. This implies that
$$
\abs{x-\lambda}>2R(\rho_n)\Rightarrow \abs{x-\lambda}-R(\rho_n)>\frac12\cdot \abs{x-\lambda}.
$$
Since $\ker d(\cdot)$ is monotone increasing, 
\begin{eqnarray*}
-p_{\mu_\lambda}(x)=\integrate{E_\lambda}{}{-\ker d\bb{\abs{x-y}}}\mu_\lambda(y)&\le& -\ker d\bb{\abs{x-\lambda}-R(\rho_n)}\\
&\ge&-\ker d\bb{\abs{x-\lambda}+R(\rho_n)}.
\end{eqnarray*}
Here our proof splits into two cases: the case where $d=2$ and the case where $d\ge 3$:
\paragraph{The case $d=2$:}
\begin{eqnarray*}
-p_{\mu_\lambda}(x)-(-\log\abs{x-\lambda})&=&\integrate{E_\lambda}{}{-\log\bb{\abs{x-y}}-\bb{-\log\abs{x-\lambda}}}\mu_\lambda(y)\\
&\le&\log\bb{\frac{\abs{x-\lambda}+R(\rho_n)}{\abs{x-\lambda}}}\le\frac{2R(\rho_n)}{\abs{x-\lambda}}\\
&\ge&\log\bb{\frac{\abs{x-\lambda}-R(\rho_n)}{\abs{x-\lambda}}}\ge-\frac{2R(\rho_n)}{\abs{x-\lambda}}.
\end{eqnarray*}
as $-2x\le\log(1-x)\le\log(1+x)\le x$ whenever $x\in\bb{0,\frac12}$.
\paragraph{The case $d\ge3$:}
\begin{eqnarray*}
-p_{\mu_\lambda}(x)&-&\bb{-\ker d\bb{\abs{x-\lambda}}}=\integrate{E_\lambda}{}{\frac1{\abs{x-y}^{d-2}}-\frac1{\abs{x-\lambda}^{d-2}}}\mu_\lambda(y)\\
&\le&\frac1{\bb{\abs{x-\lambda}-R(\abs\lambda)}^{d-2}}-\frac1{\abs{x-\lambda}^{d-2}}=\frac1{\abs{x-\lambda}^{d-2}}\bb{\bb{\frac{\abs{x-\lambda}}{\abs{x-\lambda}-R(\abs\lambda)}}^{d-2}-1}\\
&\le&\frac1{\abs{x-\lambda}^{d-2}}\bb{\bb{1+\frac{2R(\rho_n)}{\abs{x-\lambda}}}^{d-2}-1}\le \frac{2d^{3d}R(\rho_n)}{\abs{x-\lambda}^{d-1}}.
\end{eqnarray*}
Similarly,
\begin{eqnarray*}
-p_{\mu_\lambda}(x)&-&\bb{-\ker d\bb{\abs{x-\lambda}}}\ge\frac1{\bb{\abs{x-\lambda}+R(\abs\lambda)}^{d-2}}-\frac1{\abs{x-\lambda}^{d-2}}\ge -\frac{2d^{3d}R(\rho_n)}{\abs{x-\lambda}^{d-1}}.
\end{eqnarray*}

We conclude that no matter if $d=2$ or $d\ge 3$, for every $\lambda\in\Lambda$ so that $x\nin C_\lambda$
$$
\abs{-p_{\mu_\lambda}(x)-(-\ker d(\abs{x-\lambda}))}\le \frac{c_d\cdot R(\rho_n)}{\abs{x-\lambda}^{d-1}}
$$
for some uniform constant $c_d$ which depends on the dimension alone. Let $\Lambda_0=\bset{\lambda\in\Lambda,\; x\nin C_\lambda}$, then
\begin{eqnarray*}
&&\abs{u(x)-R(\rho_n)^{d-2}\underset{\lambda\in\Lambda}\sum \bb{-\ker d\bb{\abs{x-\lambda}}\indic{\R^d\setminus C_\lambda}(x)-p_{\mu_\lambda}(x)\cdot\indic{C_\lambda}(x)}}\\
&&\le R(\rho_n)^{d-2}\underset{\lambda\in\Lambda_0}\sum \abs{-p_{\mu_\lambda}(x)+\ker d\bb{\abs{x-\lambda}}}\le c_d\sumit j 2 \infty \frac{R(\rho_n)^{d-1}}{\bb{R(\rho_n)\cdot j}^{d-1}}\cdot\ell_j(x)\\
&&\le c_d\sumit j {\left\lfloor dist(x,\Lambda)\right\rfloor} {\left\lfloor dist(x,\Lambda)\right\rfloor+2m_0+1} \frac{\ell_j(x)}{j^{d-1}}\le c_d\cdot\kappa_d\sumit j {\left\lfloor dist(x,\Lambda)\right\rfloor} {\left\lfloor dist(x,\Lambda)\right\rfloor+2m_0+1} \frac{j^{d-1}}{j^{d-1}}=\alpha_d\cdot m_0,
\end{eqnarray*}
concluding the proof.
\end{proof}
\begin{cor}\label{cor:max}
Let $\alpha_d$ be the constant from Claim \ref{clm:approx}. Then
$$
0\le \underset{x\in\partial B(0,A\rho_n)}\sup\;u(x)-u(A\rho_n)\le 4\alpha_d\cdot m_0.
$$
\end{cor}
\begin{proof}
Remember that $e(\theta)\in S^{d-1}$ was chosen so that $\underset{x\in\partial B(0,A\rho_n)}\sup\;u_{\Gamma}(x)=u_{\Gamma}(A\rho_ne(\theta))$. Then
\begin{eqnarray*}
&&\underset{x\in\partial B(0,A\rho_n)}\sup\;u(x)-u(A\rho_n)\\
&&\overset{\text{Claim \ref{clm:approx}}}\le 2\alpha_d\cdot m_0+ R(\rho_n)^{d-2}\bb{\underset{x\in\partial B(0,A\rho_n)}\sup\; \underset{\lambda\in\Lambda}\sum \bb{-\ker d\bb{\abs{x-\lambda}}-\bb{-\ker d\bb{\abs{A\rho_n-\lambda}}}}}\\
&&= 2\alpha_d\cdot m_0+R(\rho_n)^{d-2}\bb{\underset{x\in\partial B(0,A\rho_n)}\sup\; \underset{\lambda\in\Lambda}\sum \bb{-\ker d\bb{\abs{e(\theta)x-e(\theta)\lambda}}-\bb{-\ker d\bb{\abs{A\rho_ne(\theta)-e(\theta)\lambda}}}}}\\
&&=2\alpha_d\cdot m_0+R(\rho_n)^{d-2}\bb{\underset{x\in\partial B(0,A\rho_n)}\sup\; \underset{\lambda\in\Gamma}\sum \bb{-\ker d\bb{\abs{e(\theta)x-\lambda}}-\bb{-\ker d\bb{\abs{A\rho_ne(\theta)-\lambda}}}}}\\
&&\overset{\text{Claim \ref{clm:approx}}}\le4\alpha_d\cdot m_0+\underset{x\in\partial B(0,A\rho_n)}\sup\; u_\Gamma(x)-u_\Gamma(A\rho_ne(\theta))=4\alpha_d\cdot m_0,
\end{eqnarray*}
concluding the proof.
\end{proof}
\subsubsection{Bounding $B_E-B_n$}
To find suitable bounds for $u$, we will use an idea similar to the one used by Carroll and Ortega-Cerd\`a in \cite{Orthega2007}. In this paper they approximate a similar function by summing approximations of it on annuli centred at the point where the function is estimated. 

We will first need to bound the distance between the point $A\rho_n$ and the set $\Lambda$.
\begin{claim}\label{clm:max_location}
There exists a constant $M_d=M_d(A)\ge 50A$, which depends on the dimension and the constant $A$, so that if $m_0\ge M_d$, then for every $\lambda\in\Lambda$ we have $\abs{A\rho_{n-1}-\lambda}\le10m_0R(\rho_n)$.
\end{claim}
\begin{proof}
Assume that it is not true, then there exists $\lambda_0\in\Lambda$ so that $\abs{A\rho_{n-1}-\lambda_0}>10m_0R(\rho_n)$. By the triangle inequality, and since $diam(\Lambda)\le2m_0$,  for every $\lambda\in\Lambda$ we have
\begin{eqnarray*}
\abs{A\rho_{n}-\lambda}&\ge&\abs{A\rho_{n-1}-\lambda}-\bb{A\rho_n-A\rho_{n-1}}>\abs{A\rho_{n-1}-\lambda_0}-\abs{\lambda-\lambda_0}-AR(\rho_{n-1})\\
&\ge&8m_0R(\rho_n)-AR(\rho_n)\ge7m_0R(\rho_n),
\end{eqnarray*}
if $m_0>A$. Let $A\rho_ne(\tau)\in\partial D_n$ be so that $dist(A\rho_ne(\tau),\Lambda)<m_0R(\rho_n)$. Such an element exists since
\begin{eqnarray*}
dist\bb{\Lambda,\partial\bb{B\bb{0,A\rho_n}}}&=&dist\bb{\Gamma,\partial\bb{B\bb{0,A\rho_n}}}\le dist\bb{\Gamma,\partial\bb{B\bb{0,A\rho_{n-1}}}}+\bb{A\rho_n-A\rho_{n-1}}\\
&\le& R(\rho_n)\bb{\frac{m_0}{10}+10+A}<m_0R(\rho_n),
\end{eqnarray*}
for $m_0$ numerically large enough. Using Claim \ref{clm:approx}, and the fact that $\ker d(t)$ is a monotone non-decreasing function, 
\begin{eqnarray*}
u(A\rho_ne(\tau))-u(A\rho_n)&\ge& R(\rho_n)^{d-2}\underset{\lambda\in\Lambda}\sum \bb{-\ker d\bb{\abs{A\rho_ne(\tau)-\lambda}}-\bb{-\ker d\bb{\abs{A\rho_n-\lambda}}}}-2\alpha_d\cdot m_0\\
&\ge&R(\rho_n)^{d-2}\underset{\lambda\in\Lambda}\sum \bb{\ker d\bb{7m_0R(\rho_n)}-\ker d\bb{3m_0R(\rho_n)}}-2\alpha_d\cdot m_0\\
&\ge&\beta_d\cdot m_0^2-2\alpha_d\cdot m_0=\beta_d\cdot m_0^2\bb{1-\frac{\alpha_d}{\beta_d\cdot m_0}}\ge\frac{\beta_d}2\cdot m_0^2,
\end{eqnarray*}
for some constant $\beta_d>0$, as long as $m_0$ is numerically large enough. Then, following Corollary \ref{cor:max},
\begin{eqnarray*}
\underset{x\in\partial B(0,A\rho_n)}\sup\;u(x)&\le& u(A\rho_n)+4\alpha_d\cdot m_0\le u(A\rho_ne(\tau))-\frac{\beta_d}2\cdot m_0^2+4\alpha_d\cdot m_0\\
&<&u(A\rho_ne(\tau))-\frac{\beta_d}4\cdot m_0^2<\underset{x\in\partial B(0,A\rho_n)}\sup\;u(x)-\frac{\beta_d}4\cdot m_0^2
\end{eqnarray*}
as long as $m_0$ is large enough. We arrive at a contradiction concluding the proof.
\end{proof}
\begin{claim}\label{clm:B_E-B_n}
There exists a constant $\gamma_d$ which depends on the dimension alone, so that if $m_0\ge 50A$, then
$$
B_E-B_n\le\gamma_d\bb{-\ker d\bb{\eps(\rho_n)}+m_0^2}.
$$
\end{claim}
\begin{proof}
Let $x\in E_\nu$ for some $\nu\in\Lambda$. Then, by Claim \ref{clm:max_location}
\[
\abs{A\rho_n-x}\le \abs{A\rho_{n-1}-\nu}+R(\rho_n)+\bb{A\rho_n-A\rho_{n-1}}\le 10m_0R(\rho_n)+R(\rho_n)+AR(\rho_n)\le 12m_0R(\rho_n),\label{eq:A_rho_n-x}\tag{$\clubsuit$}
\]
provided that $m_0>A$. Following Frostman's theorem,
\begin{eqnarray*}
u(x)-u(A\rho_n)&\overset{\text{Claim \ref{clm:approx}}} \le& R(\rho_n)^{d-2}\bb{-\ker d\bb{R(\rho_n)\eps(\rho_n)}+\ker d\bb{\abs{A\rho_n-\nu}}}\\
&&+R(\rho_n)^{d-2}\underset{\lambda\in\Lambda\setminus\bset\nu}\sum\bb{-\ker d\bb{\abs{x-\lambda}}-\bb{-\ker d\bb{\abs{A\rho_n-\lambda}}}}+2\alpha_d\cdot m_0\\
&&\overset{\text{by (\ref{eq:A_rho_n-x})}}\le-\ker d\bb{\eps(\rho_n)}+\ker d\bb{12m_0R(\rho_n)}+2\alpha_d\cdot m_0\\
&&+R(\rho_n)^{d-2}\underset{\lambda\in\Lambda\setminus\bset\nu}\sum\frac d{ds}\ker d(s_\lambda)\bb{\abs{A\rho_n-\lambda}-\abs{x-\lambda}}.
\end{eqnarray*}
To bound the sum, note that for every $\lambda\in\Lambda\subset B\bb{0,A\rho_{n-1}-\frac{m_0R(\rho_n)}{10}}$
$$
\abs{A\rho_n-\lambda}\ge \abs{A\rho_{n-1}-\lambda}-\bb{A\rho_n-A\rho_{n-1}}\ge \frac{m_0R(\rho_n)}{10}-AR(\rho_n)\ge\frac{\abs{x-\lambda}}{25},
$$
if $m_0\ge 50A$, implying that
$$
s_\lambda\ge \min\bset{\abs{A\rho_n-\lambda},\abs{x-\lambda}}\ge\frac{\abs{x-\lambda}}{25}.
$$
Then
\begin{eqnarray*}
&&\hspace{-2em}R(\rho_n)^{d-2}\underset{\lambda\in\Lambda\setminus\bset\nu}\sum\frac d{ds}\ker d(s_\lambda)\bb{\abs{A\rho_n-\lambda}-\abs{x-\lambda}}\le R(\rho_n)^{d-2}\underset{\lambda\in\Lambda\setminus\bset\nu}\sum\frac {\max\bset{1,d-2}}{s_\lambda^{d-1}}\cdot\abs{A\rho_n-x}\\
&\overset{\text{by (\ref{eq:A_rho_n-x})}}\le& 12m_0R(\rho_n)^{d-1}\underset{\lambda\in\Lambda\setminus\bset\nu}\sum\frac {25d}{\abs{x-\lambda}}\le 300m_0\cdot d\cdot R(\rho_n)^{d-1}\sumit j 1 {2m_0}\frac {\ell_j(x)}{\bb{j\cdot R(\rho_n)}^{d-1}}\\
&\le& 300\cdot m_0\cdot d\cdot\kappa_d\sumit j 1 {2m_0}1=600d\cdot\kappa_dm_0^2.
\end{eqnarray*}
Overall, using Corollary \ref{cor:max}, we get that 
\begin{eqnarray*}
&&\hspace{-2em}B_E-B_n\le \underset{x\in\underset{\lambda\in\Lambda}\bigcup E_\lambda}\sup\; u(x)-\underset{y\in\partial B(0,A\rho_n)}\sup\; u(y)\le\underset{x\in\underset{\lambda\in\Lambda}\bigcup E_\lambda}\sup\; u(x)-u(A\rho_n)+4\alpha_d\cdot m_0\\
&&\hspace{-2em}\le-\ker d\bb{\eps(\rho_n)}+\ker d\bb{12m_0R(\rho_n)}+2\alpha_d\cdot m_0+600d\cdot\kappa_dm_0^2+4\alpha_d\cdot m_0\le\gamma_d\bb{-\ker d\bb{\eps(\rho_n)}+m_0^2},
\end{eqnarray*}
as needed.
\end{proof}
\subsubsection{Bounding the harmonic measure}
Let $\tilde E=\underset{\lambda\in\Lambda}\bigcup E_\lambda$. Following the maximum principle
$$
\omega(A\rho_{n-1},\tilde E,D_n\setminus\tilde E)\le\omega(A\rho_{n-1}, E,D_n\setminus E).
$$
On the other hand, by the same principle, for every $x\in D_n\setminus\tilde E$ we know that
$$
\frac{u(x)-B_n}{B_E-B_n}\le\omega(x,\tilde E,D_n\setminus\tilde E).
$$
To conclude the proof, it is left to bound $u(A\rho_{n-1})-B_n$ from bellow.
\begin{claim}\label{clm:low_bnd}
If $A$ is chosen numerically large enough, then there exists a constant $\chi_d>0$ which depends on the dimension alone, so that
$$
u(A\rho_{n-1})-B_n\ge \chi_d\cdot m_0.
$$
\end{claim}
\begin{proof}
Following Corollary \ref{cor:max}, we know that
$$
u(A\rho_{n-1})-B_n\ge u(A\rho_{n-1})-u(A\rho_n)-4\alpha_dm_0.
$$
Next, using Claim \ref{clm:approx} it is enough to bound from bellow 
$$
R(\rho_n)^{d-2}\underset{\lambda\in\Lambda}\sum\bb{-\ker d\bb{\abs{A\rho_{n-1}-\lambda}}-\bb{-\ker d\bb{\abs{A\rho_n-\lambda}}}}.
$$
Now, for every $\lambda\in\Lambda$, 
\begin{eqnarray*}
&&\hspace{-2em}\abs{A\rho_n-\lambda}^2=\bb{A\rho_n-\lambda_1}^2+\sumit j 2 d \lambda_j^2=\bb{A\rho_{n-1}-\lambda_1}^2+\sumit j 2 d \lambda_j^2+\bb{A\rho_{n}-\lambda_1}^2-\bb{A\rho_{n-1}-\lambda_1}^2\\
&&\hspace{-2em}\ge\abs{A\rho_{n-1}-\lambda}^2+AR(\rho_{n-1})\bb{A\rho_n+A\rho_{n-1}-2\lambda_1}\ge \abs{A\rho_{n-1}-\lambda}^2+2AR(\rho_{n-1})\bb{A\rho_{n-1}-\lambda_1}.
\end{eqnarray*}
Using the mean-value theorem
\begin{eqnarray*}
R(\rho_n)^{d-2}&&\hspace{-1.5em}\underset{\lambda\in\Lambda}\sum\bb{-\ker d\bb{\abs{A\rho_{n-1}-\lambda}}-\bb{-\ker d\bb{\abs{A\rho_n-\lambda}}}}\\
=R(\rho_n)^{d-2}&&\hspace{-1.5em}\underset{\lambda\in\Lambda}\sum\frac d{ds} \ker d(s_\lambda)\bb{\abs{A\rho_n-\lambda}-\abs{A\rho_{n-1}-\lambda}}\\
\ge R(\rho_n)^{d-2}&&\hspace{-1.5em}\underset{\lambda\in\Lambda}\sum\frac 1{\abs{A\rho_n-\lambda}^{d-1}}\cdot\frac{\abs{A\rho_n-\lambda}^2-\abs{A\rho_{n-1}-\lambda}^2}{\abs{A\rho_n-\lambda}+\abs{A\rho_{n-1}-\lambda}}\ge \underset{\lambda\in\Lambda}\sum\frac {AR(\rho_n)^{d-1}}{\abs{A\rho_n-\lambda}^{d-1}}\cdot\frac{A\rho_{n-1}-\lambda_1}{\abs{A\rho_{n-1}-\lambda}},
\end{eqnarray*}
following the calculation above. Next, since $\lambda\in B\bb{0,A\rho_{n-1}-\frac{m_0R(\rho_n)}{10}}$, then 
$$
A\rho_{n-1}-\lambda_1\ge A\rho_{n-1}-\abs\lambda> A\rho_{n-1}-\bb{A\rho_{n-1}-\frac{m_0R(\rho_n)}{10}}= \frac{m_0R(\rho_n)}{10},
$$
while following Claim \ref{clm:max_location}, for every $\lambda\in\Lambda$ we have $\abs{A\varrho_{n-1}-\lambda}\le 10m_0R(\rho_n)$. We conclude that
\begin{eqnarray*}
R(\rho_n)^{d-2}&&\hspace{-1.5em}\underset{\lambda\in\Lambda}\sum\bb{-\ker d\bb{\abs{A\rho_{n-1}-\lambda}}-\bb{-\ker d\bb{\abs{A\rho_n-\lambda}}}}
\ge\cdots\ge\underset{\lambda\in\Lambda}\sum\frac {AR(\rho_n)^{d-1}}{\abs{A\rho_n-\lambda}^{d-1}}\cdot\frac{A\rho_{n-1}-\lambda_1}{\abs{A\rho_{n-1}-\lambda}}\\
&\ge&\underset{\lambda\in\Lambda}\sum\frac {AR(\rho_n)^{d-1}}{\abs{A\rho_n-\lambda}^{d-1}}\cdot\frac{\frac{m_0R(\rho_n)}{10}}{10m_0R(\rho_n)}
\ge \frac A{100}\sumit j {\frac{m_0}{4d}}{\frac{3m_0}{4d}}\frac {R(\rho_n)^{d-1}\ell_j(A\rho_n)}{\bb{(j+1)d\cdot R(\rho_n)}^{d-1}}\\
&\ge&\frac A{d^d\cdot 100\kappa_d}\sumit j {\frac{m_0}{4d}}{\frac{3m_0}{4d}}\bb{\frac j{j+1}}^{d-1}\ge m_0\cdot\frac A{(2d)^d\cdot100\kappa_d}.
\end{eqnarray*}
Combining everything together we get that\begin{eqnarray*}
u(A\rho_{n-1})-B_n&\ge& u(A\rho_{n-1})-u(A\rho_n)-4\alpha_dm_0\\
&\ge& R(\rho_n)^{d-2}\underset{\lambda\in\Lambda}\sum\bb{-\ker d\bb{\abs{A\rho_{n-1}-\lambda}}-\bb{-\ker d\bb{\abs{A\rho_n-\lambda}}}}-6\alpha_dm_0\\
&\ge& m_0\cdot\frac A{(2d)^d\cdot100\kappa_d}-6\alpha_dm_0= m_0\bb{\frac A{(2d)^d\cdot100\kappa_d}-6\alpha_d}\ge\chi_d\cdot m_0,
\end{eqnarray*}
for some $\chi_d>0$ which depends on the dimension alone, as long as $A$ is chosen large enough.
\end{proof}

Following Claim \ref{clm:low_bnd},
$$
\omega(A\rho_{n-1}, E,D_n\setminus E)\ge \omega(A\rho_{n-1},\tilde E,D_n\setminus\tilde E)\ge \frac{u(A\rho_{n-1})-B_n}{B_E-B_n}\ge \frac{\chi_d\cdot m_0}{\gamma_d\bb{-\ker d(\eps(\rho_n))+ m_0^2}}.
$$
To choose the parameters correctly, note that $\alpha_d$ is a numerical constant which depends on the dimension alone. Next, we choose $A$ so that Claim \ref{clm:low_bnd} holds. The constant $A$ determines both $\chi_d$ and the lower bound $M_d$ appearing in Claim \ref{clm:max_location}. To conclude the proof, let $i(n)$ be the term appearing in Claim \ref{clm:R_n/R_k} and define $m_0=\max\bset{i(n), M_d}$. Then
$$
\omega(A\rho_{n-1}, E,D_n\setminus E)\ge\cdots\ge \frac{\chi_d\cdot m_0}{\gamma_d\bb{-\ker d(\eps(\rho_n))+ m_0^2}}\ge \frac{c}{\sqrt{-\ker d(\eps(\rho_n))}}.
$$
\section{A subharmonic function- the proof of Theorem \ref{thm:function}\ref{thm:functionB}}\label{sec:func}
In this section we will construct a subharmonic function, for which the set $\bset{u\le 0}$ is $(\eps,R)$-recurrent and 
$$
\log M_u(\rho):=\log\bb{\underset{\bset{\abs x=\rho}}\max\; u(x)}\lesssim \integrate 1\rho{\varphi(t)}t,
$$
thus proving Theorem \ref{thm:function}\ref{thm:functionB}.
\subsection{The function:}
Let $R_0(t),\eps_0(t)$ be two functions that will be chosen at the end of the proof. Similarly to Subsection \ref{subsec:not}, we define the sequence
$$
\rho_0:=0,\; \rho_{n+1}=\rho_n+R_0\bb{\rho_n},
$$
and let $r_0$ be so that $R_0(t)\le \frac t2$ for all $t\ge r_0$. For every $k\in\N,\; \rho_k\ge r_0$, let $\Lambda_k$ be a collection of points on $2\rho_kS^{d-1}$, the sphere of radius $2\rho_k$, satisfying
\begin{enumerate}
 \item For every $\mu\neq\lambda\in\Lambda_k,\;\; B(\lambda,R_0(\abs\lambda)+1)\cap B(\mu,R_0(\abs\mu)+1)=\emptyset$.
 \item $2\rho_kS^{d-1}\subset\underset{\lambda\in\Lambda_k}\bigcup B(\lambda,4R_0(\abs\lambda))$.
\end{enumerate}
We allow freedom to choose this collection, and it is clear that such a collection exists. Let $\Lambda:=\underset{ k, \rho_k\ge r_0}\bigcup\Lambda_k$ for $\Lambda_k$ defined above, and define the function:
$$
v(x)=\exp\bb{C\integrate 1{\abs x }{\varphi(t)}t},
$$
where the constant $C>0$ will be chosen later. Note that this function is radial, and 
\begin{eqnarray*}
\Delta v(x)&=&\frac{\partial^2 v}{\partial r^2}(x)+\frac{d-1}{\abs x }\cdot\frac{\partial v}{\partial r}(x)\ge Cv(x)\varphi^2(\abs x )\bb{C-\frac d{dt}\bb{\frac1{\varphi}}(\abs x )}.
\end{eqnarray*}
Since $\frac1{\varphi}$ is bounded, for large enough $C$ the function $v$ is subharmonic.
For every $\lambda\in\Lambda$ we let $v_\lambda(\xi)=v(\lambda+R_0(\abs\lambda)\cdot \xi)$, and define the function
$$
w(x)=		\begin{cases}
		P_{B}v_\lambda\bb{\frac{x-\lambda}{R_0(\abs\lambda)}}+A_\lambda\cdot \widetilde {\ker d}\bb{\frac{\abs{x-\lambda} }{ R_0(\abs\lambda) }}&, \lambda\in\Lambda\text{ and } x\in B(\lambda, R_0(\abs\lambda) )\\
		v(x)&, \text{ otherwise}
		\end{cases},
$$
where $P_B$ denotes Poisson integral of the set $B=B(0,1)=\bset{x,\abs x<1}$ and 
$$
\widetilde {\ker d}:=\begin{cases}
		\ker d(x)&, d=2\\
		1+\ker d(x)&, d\ge 3
		\end{cases}.
$$
This function should satisfy two conditions:
\begin{enumerate}
\item For every $\lambda$, $x\in B(\lambda, R_0(\abs\lambda) \cdot \eps_0(\abs\lambda) )\Rightarrow w(x)\le 0$.
\item $w$ is subharmonic.
\end{enumerate}
\paragraph{The first condition:}~\\
To address the first condition, let $\rho$ denote the maximal radius where $w|_{B(\lambda,\rho)}\le 0$. Using the maximum principle:
\begin{eqnarray*}
\underset{\abs{\lambda-x} =\rho}\max\;w(x)&=&\underset{\abs{\lambda-x} =\rho}\max\; P_{B}v_\lambda\bb{\frac{x-\lambda}{R_0(\abs\lambda)}}+A_\lambda\widetilde {\ker d}\bb{\frac{\abs{\lambda-x} }{ R_0(\abs\lambda) }}\\
&\le& \underset{\abs \xi=1}\max\; v_\lambda(\xi)+A_\lambda\widetilde {\ker d} \bb{\frac\rho{ R_0(\abs\lambda) }}= v\bb{\abs \lambda +R_0(\abs\lambda) }+A_\lambda\widetilde {\ker d} \bb{\frac\rho{ R_0(\abs\lambda) }}\le 0\\
&\iff &\frac{\rho}{R_0(\abs\lambda)}\le \bb{\widetilde {\ker d}}\inv\bb{-\frac{v\bb{\abs \lambda +R_0(\abs\lambda)}}{A_\lambda}}.
\end{eqnarray*}
\paragraph{The second condition:}~\\
To address the subharmonicity condition, we note that by the way we defined the function $w$, it is subharmonic on $\R^d\setminus\underset{\lambda\in\Lambda}\bigcup\partial B\bb{\lambda, R_0(\abs\lambda) }$. To show it is subharmonic in $\R^d$ we will use without proof a glueing argument, followed by  Poisson-Jenssen's formula:
\begin{claim}\label{clm:glueing}
Let $\Omega\subseteq\R^d$ be a domain, and let $\Omega_1\Subset\Omega$ be a subdomain, so that $\partial\Omega_1$ is an orientable, smooth, $d-1$ manifold. Every function $u$ which is continuous on $\Omega$ and subharmonic on $\Omega\setminus\partial\Omega_1$, is subharmonic on $\Omega$ if on $\partial\Omega_1$ it satisfies
$$
\frac{\partial u}{\partial n_1}\le \frac{\partial u}{\partial n_2},
$$
where $n_1$ is the outer normal to $\Omega_1$ along $\partial\Omega_1$ and $n_2$ is the outer normal to $\Omega\setminus\Omega_1$ along $\partial\Omega_1$. 
\end{claim}
We will use this claim with $\Omega=\R^d$, and $\Omega_1=B(\lambda,R_0(\abs\lambda))$, where $\frac{\partial w}{\partial n_2}$ is known. To 
bound $\frac{\partial w}{\partial n_1}$ we will use Poisson-Jenssen's formula. The problem is then reduced to showing that $\forall\xi\in S^{d-1}$:
\begin{eqnarray*}
&&\frac{\partial w}{\partial n_1}(\lambda+R_0(\abs\lambda)\xi)=\limit r {1^-}\frac{v\bb{\lambda+ R_0(\abs\lambda) \xi}-P_{B}v_\lambda(r \xi)}{1-r}+A_\lambda\;\max\bset{1,d-2}\\
&&=\frac{\partial v}{\partial n}\bb{\lambda+ R_0(\abs\lambda) \xi}+\limit r {1^-}\frac{G_{B}\bb{r \xi}}{1-r}+A_\lambda\;\max\bset{1,d-2}\\
&&=\frac{\partial v}{\partial n}\bb{\lambda+ R_0(\abs\lambda) \xi}-\frac{\partial G_{B}}{\partial r}\bb{\xi}+A_\lambda\;\max\bset{1,d-2}
\le \frac{\partial w}{\partial n_2}(\lambda+R_0(\abs\lambda)\xi)=\frac{\partial v}{\partial n}\bb{\lambda+ R_0(\abs\lambda) \xi},
\end{eqnarray*}
where if $g_B$ denote Green's function for the unit ball, $B=B(0,1)$, then
$$
G_{B}(\xi):=c_d\integrate {B}{}{g_{B}(\xi,y)\Delta v_\lambda(y)}m_d(y), \;  m_d \text{ is Lebegue's measure in }\R^d.
$$
We conclude that in order to show subharmonicity, it is enough to restrict $A_\lambda$ so that
$$
A_\lambda\max\bset{1,d-2}\le \frac{\partial G_{B}}{\partial r}\bb{\xi},\;\forall\xi\in S^{d-1}.
$$
To bound $\frac{\partial G_{B}}{\partial r}\bb{\xi}$ from bellow, note that since the collection $\bset{\frac{\partial g_{B}}{\partial n}\bb{\xi,\cdot}, \xi\in S^{d-1}}$ is uniformly integrable, we may swap the integral and the derivative to obtain
$$
\frac{\partial G_{B}}{\partial r}\bb{\xi}=\frac{\partial}{\partial r}\bb{c_d\integrate {B}{}{g_{B}(\xi,y)\Delta v_\lambda(y)}m_d(y)}=c_d\integrate {B}{}{\frac{\partial g_{B}(\cdot,y)}{\partial r}(\xi)\Delta v_\lambda(y)}m_d(y).
$$
Now for every $y\in B(0,1)$,
\begin{eqnarray*}
\Delta v_\lambda(y)&=&Cv(\lambda+R_0(\abs\lambda) y)\varphi^2(\abs {\lambda+R_0(\abs\lambda) y} )R_0(\abs\lambda)^2\cdot\bb{C-\frac d{dy}\bb{\frac1{\varphi}}(\abs {\lambda+R_0(\abs\lambda) y} )}\\
&&\ge\frac{C^2}2v(\abs \lambda -R_0(\abs\lambda) )R_0(\abs\lambda)^2\varphi^2(\abs \lambda +R_0(\abs\lambda)),
\end{eqnarray*}
if one chooses $C$ so that $\frac d{dt}\bb{\frac1{\varphi}}<\frac C2$. Then
\begin{eqnarray*}
\frac{\partial G_{B}}{\partial r}\bb{\xi}&\ge&\cdots\ge\frac{C^2}2v(\abs \lambda -R_0(\abs\lambda) )R_0(\abs\lambda)^2\varphi^2(\abs \lambda +R_0(\abs\lambda))c_d\integrate {B}{}{\frac{\partial g_{B}(\cdot,y)}{\partial r}(\xi)}m_d(y)\\
&\gtrsim& C^2\cdot\frac{v(\abs \lambda + R_0(\abs\lambda) )}{-\ker d(\eps(\abs\lambda) )}\cdot \exp\bb{\frac{-\Theta(1)C}{\sqrt{-\ker d\bb{ \eps(\abs\lambda) }}}},
\end{eqnarray*}
\subsection{Choosing $A_\lambda$.}
In order to get a maximal radius $r_\lambda$ where $w|_{B(\lambda,r_\lambda)}\le0$, we need to choose $A_\lambda$ to be as big as possible, so that $w$ is subharmonic, i.e
$$
A_\lambda\max\bset{d-2,1}= A_1\cdot\frac{C^2v(\abs \lambda + R_0(\abs\lambda) )}{-\ker d(\eps(\abs\lambda) )}\exp\bb{\frac{-A_2C}{\sqrt{-\ker d\bb{ \eps(\abs\lambda) }}}},
$$
for suitable constants $A_1,A_2$. We then get that the maximal radius where $w|_{B(\lambda,r_\lambda)}\le 0$ is
\begin{eqnarray*}
\frac{r_\lambda}{ R_0(\abs\lambda) }&=&\bb{\widetilde {\ker d}}\inv\bb{-\frac{v\bb{\abs \lambda +R_0(\abs\lambda) }}{A_\lambda}}\\
&=&\bb{\widetilde {\ker d}}\inv\bb{\exp\bb{\frac{A_2C}{\sqrt{-\ker d\bb{ \eps_0(\abs\lambda) }}}}\frac{\max\bset{1,d-2}\ker d(\eps_0(\abs\lambda) )}{A_1C^2}}\\
&\ge&\bb{ \ker d}\inv\bb{\ker d(\eps_0(\abs\lambda) )\bb{\frac{\max\bset{1,d-2}\exp\bb{\frac{A_2C}{\sqrt{-\ker d\bb{ \eps_0(\abs\lambda) }}}}}{A_1C^2}-\frac1{\ker d(\eps_0(\abs\lambda) )}}}.
\end{eqnarray*}
Now, remember that $\ker d,\;\ker d\inv$ are monotone increasing functions, and $\ker d(t)<0$ for $t<1$. Since $\eps_0(t)<1$ for all $t\in(0,\infty)$, then $\ker d(\eps_0(\abs\lambda))<0$ and so
\begin{eqnarray*}
\frac{r_\lambda}{ R_0(\abs\lambda) }&\ge&\cdots\ge\bb{ \ker d}\inv\bb{\ker d(\eps_0(\abs\lambda) )\bb{\frac{\max\bset{1,d-2}\exp\bb{\frac{A_2C}{\sqrt{-\ker d\bb{ \eps_0(\abs\lambda) }}}}}{A_1C^2}-\frac1{\ker d(\eps_0(\abs\lambda) )}}}\\
&\ge&\bb{ \ker d}\inv\bb{\ker d(\eps_0(\abs\lambda) )\cdot\frac{\max\bset{1,d-2}\exp\bb{\frac{A_2C}{\sqrt{-\ker d\bb{ \eps_0(\abs\lambda) }}}}}{A_1C^2}}\\
&\ge&\bb{ \ker d}\inv\bb{\ker d(\eps_0(\abs\lambda) )}=\eps_0(\abs\lambda),
\end{eqnarray*}
as long as
\begin{eqnarray*}
\frac{\max\bset{1,d-2}\exp\bb{\frac{A_2C}{\sqrt{-\ker d\bb{ \eps_0(\abs\lambda) }}}}}{A_1C^2}\le1,
\end{eqnarray*}
which holds for every $\abs\lambda $ large enough, depending on $C$, and the constants $A_1,\; A_2$.

We will conclude the proof in two steps: Let $r_1\ge r_0$ be so that for every $\abs\lambda\ge r_1$ the above holds, and $v|_{ r_1\cdot S^{d-1}}$ is radial, or in other words, $\underset k\min\; \abs{ r_1-2\rho_k}-\tilde R(\rho_k)>0$.\\
{\bf Step 1:} Let $\sigma_d=\min\bset{\frac12,Vol_d(B(0,1))}$, and define the functions
$$
s(x)=\cosh\bb{\frac{\pi\sqrt{d-1}}{\sigma_dR(0)}\cdot x_1}\prodit j 2 d \cos\bb{\frac{\pi x_j}{\sigma_dR(0)}}\indic{\bset{\abs{x_j}<\frac{\sigma_d\cdot R(0)}2}}(x),
$$
and 
$$
v_1(x)=	\begin{cases}
			\max\bset{C_1\widetilde {\ker d}\bb{\frac{\abs x }{ r_1-\sigma_dR_0(0)}},s(x)}&, \abs x< r_1\\
			C_1\widetilde {\ker d}\bb{\frac{\abs x }{ r_1-\sigma_dR_0(0)}}&, \abs x\ge r_1
		\end{cases},
$$
where $C_1$ is chosen so that
$$
C_1\widetilde {\ker d}\bb{\frac{r_1}{ r_1-\sigma_dR_0(0)}}\ge e^{r_1}\ge s(x),\;\text{ for all }\abs x=r_1.
$$
This function is subharmonic as a maximum between subharmonic functions, and $v_1(0)\ge s(0)=1$.\\
{\bf Step 2:} Define the function
$$
u(x):=	\begin{cases}
		v_1(x)&, \abs x \le r_1\\
		C_2\cdot w(x)-C_3&,\text{otherwise}
		\end{cases}
$$
where $C_2$ is chosen so that 
\begin{eqnarray*}
\left.\frac{\partial u}{\partial r^-}\right|_{\bset{\abs x = r_1}}&=&C_2\cdot \left.\frac{\partial w}{\partial r}\right|_{\bset{\abs x = r_1}}\ge\frac{C_1\cdot \max\bset{1,d-2}\cdot \bb{r_1-\sigma_dR_0(0)}^{d-2}}{ r_1^{d-1}}\\
&=&\left.\frac{\partial }{\partial r}\bb{C_1\ker d\bb{\frac{\abs x }{ r_1-\sigma_dR_0(0)}}}\right|_{\bset{\abs x =r_1}}=\left.\frac{\partial }{\partial r}\bb{v_1(x)}\right|_{\bset{\abs x =r_1}}=\left.\frac{\partial }{\partial r^+}\bb{u(x)}\right|_{\bset{\abs x =r_1}},
\end{eqnarray*}
while $C_3$ is chosen so that $\bb{C_2\cdot v(x)-C_3}|_{\bset{\abs x = r_1}}=v_1(x)|_{\bset{\abs x=r_1}}$, which is possible since on $ r_1\cdot S^{d-1}$ both functions are radial and strictly increasing. Following using Claim \ref{clm:glueing}, $u(x)$ is a subharmonic function, while $u(0)\ge s(0)=1$.
\subsection{Choosing the functions $R_0,\;\eps_0$}
To conclude the proof it is left to choose the functions $R_0,\;\eps_0$ so that for every $x\in\R^d$ we have 
$$
\frac{m_d\bb{B(x,R(\abs x ))\cap Z_u}}{m_d\bb{B(x, R(\abs x ))}}>m_d(B(x,\eps(\abs x ))).
$$
\begin{center}
\includegraphics[width=0.5\textwidth]{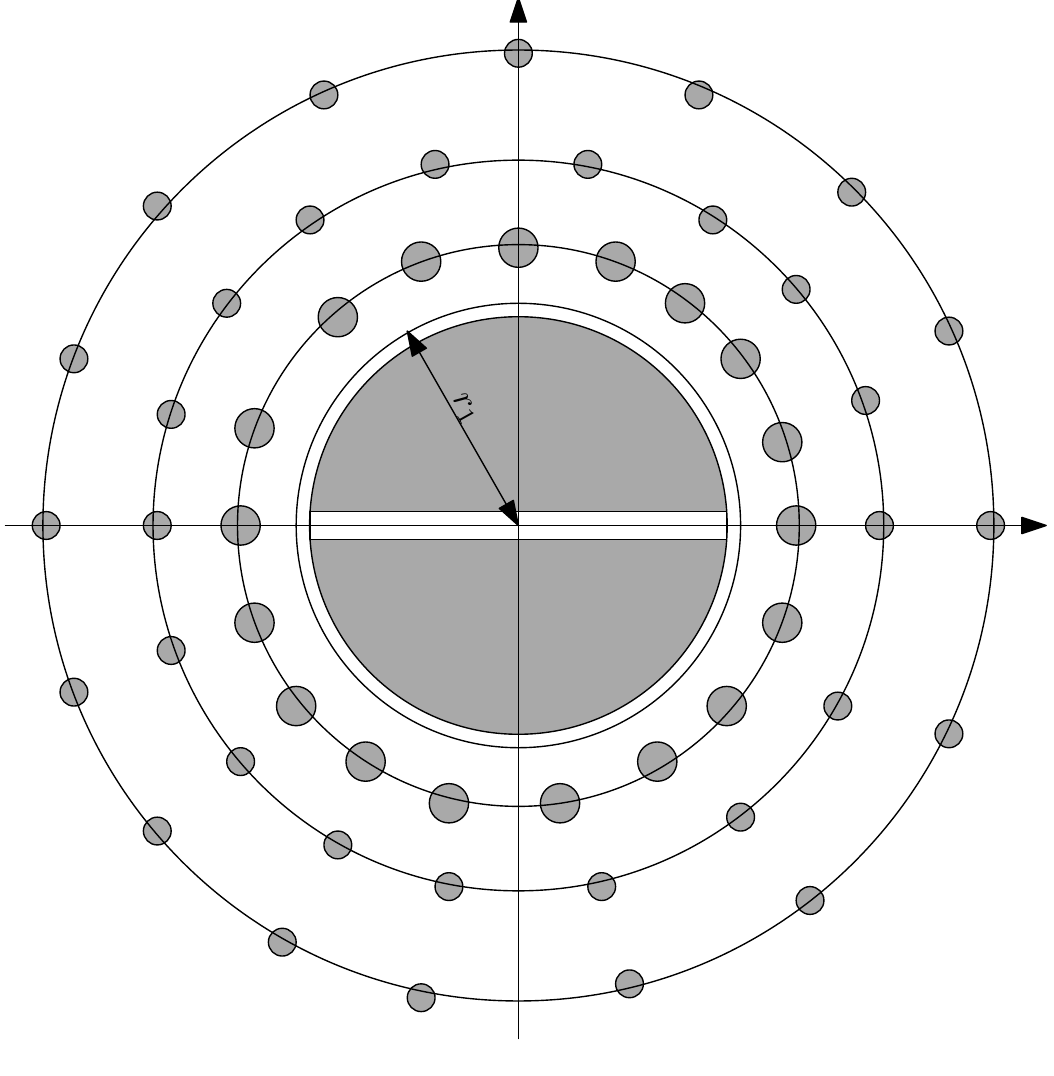}
\captionof{figure}{The zero set of $u$: The grey area is where $\bset{u= 0}$.}
\end{center}

Let $x\in\R^d$. If $\abs x<r_1$, then as long as $R_0(t)\le R(t)$,
\begin{eqnarray*}
\frac{m_d\bb{B(x,R(\abs x ))\cap Z_u}}{m_d\bb{B(x, R(\abs x ))}}&\ge& \frac{m_d(B(x,R(\abs x))-diam(B(x,R(\abs x))\cdot\bb{\sigma_dR(0)}^{d-1}}{m_d\bb{B(x, R(\abs x ))}}\\
&\ge&1-\frac12\cdot\bb{\frac{R(0)}{R(\abs x)}}^{d-1}>\frac12>\eps(\abs x).
\end{eqnarray*}
Otherwise, let $k$ be so that $2\rho_k\le\abs x<2\rho_{k+1}$, then since $\partial\bb{2\rho_kS^{d-1}}\subset\underset{\lambda\in\Lambda_k}\bigcup B(\lambda,4R(\abs\lambda))$,
$$
dist(x,\Lambda_k)\le 4R_0(\rho_k)+2\bb{\rho_{k+1}-\rho_k}=6R_0(\rho_k),
$$
and there exists $\lambda\in\Lambda_k$ so that $dist(x,\Lambda_k)\le\abs{x-\lambda}<6R_0(\rho_k)$. By inclusion
\begin{eqnarray*}
\frac{m_d\bb{B(x,R(\abs x ))\cap Z_u}}{m_d\bb{B(x, R(\abs x ))}}&\ge& \frac{m_d\bb{B(\lambda,R(\abs x)-6R_0(\rho_k))\cap Z_u}}{m_d\bb{B(x, R(\abs x ))}}\\
=\frac{m_d\bb{B(\lambda,R_0(\rho_k))\cap Z_u}}{m_d\bb{B(\lambda, R_0(\abs\lambda ))}}\cdot\frac{m_d\bb{B(\lambda, R_0(\abs\lambda ))}}{m_d\bb{B(x, R(\abs x ))}}&\ge& {m_d\bb{B(\lambda, \eps_0(\abs\lambda ))}}\bb{\frac{R_0(\rho_k)}{R(\rho_{k+1})}}^d\ge m_d\bb{B(x, \eps(\abs x ))},
\end{eqnarray*}
if we let $R_0(t)=\frac{R(t)}7$ and $\eps_0(t)=7\eps\bb{\frac t2}>\eps(t)$.

The function $u$ then satisfies all the requirements of Theorem \ref{thm:function}\ref{thm:functionB}, concluding the proof.
\section{Appendix- the proof of Proposition \ref{prop:layers}}
In order to prove this proposition, we will need the following Observation:
\begin{obs}\label{obs:independent}
Let $D$ be a domain, $F\subset D$ be some set, and let $D'\subset D$ be a subdomain. Define $F':=D'\cap F$. Then for every $x\in D'$:
\begin{eqnarray*}
\underset{y\in\partial D'}\inf\;\omega\bb{y,\partial D;D\setminus F}\le\frac{\omega\bb{x,\partial D;D\setminus F}}{\omega\bb{x,\partial D';D'\setminus F'}}\le\underset{y\in\partial D'}\sup\;\omega\bb{y,\partial D;D\setminus F}.
\end{eqnarray*}
\end{obs}
\begin{proof}
let
$$
m=\underset{y\in\partial D'}\inf\;\omega\bb{y,\partial D;D\setminus F},\;\; M=\underset{y\in\partial D'}\sup\;\omega\bb{y,\partial D;D\setminus F},
$$
then for every $x\in\partial D'$
$$
m\le\omega\bb{x,\partial D;D\setminus F}\le M.
$$
Since on $\partial D'$ we have $\omega\bb{x,\partial D';D'\setminus F'}=1$, we conclude that
$$
m\cdot \omega\bb{x,\partial D';D'\setminus F'}\le\omega\bb{x,\partial D;D\setminus F}\le M\cdot \omega\bb{x,\partial D';D'\setminus F'}
$$
for every $x\in\partial D'$. Next, for every $x\in F'$ we know that the above inequality holds as well since all the components in the inequality are zero. Then the inequality holds for every $x\in\partial\bb{D'\setminus F}$ and following the maximum principle for every  $x\in D'$.
\end{proof}
\paragraph{Proof of Proposition \ref{prop:layers}}
\begin{proof}
We will use the observation recursively beginning with the sets:
$$
F=V,\; D=D_{n+1},\;\; D'=D_n.
$$
As these sets satisfy the assumptions of Observation \ref{obs:independent}, 
\begin{eqnarray*}
\underset{x\in\partial D_n}\inf\;\omega\bb{x,\partial D_{n+1};D_{n+1}\setminus V}\le \frac{\omega\bb{0,\partial D_{n+1} ;D_{n+1} \setminus V}}{\omega\bb{0,\partial D_n;D_n\setminus V}}\le \underset{x\in\partial D_n}\sup\;\omega\bb{x,\partial D_{n+1};D_{n+1}\setminus V}.
\end{eqnarray*}
We continue to apply Observation \ref{obs:independent} over and over with the sets 
$$
D=D_{n-k+1},\; D'=D_{n-k}, F=V\cap D_{n-k+1}.
$$
Doing so recursively, we obtain the lower bound:
\begin{eqnarray*}
\omega\bb{0,\partial D_{n+1} ;D_{n+1} \setminus V}&\ge&\prodit k 1 {n}	\bb{1-\underset{x\in\partial D_k}\sup\;\omega\bb{x,V;D_{k+1}\setminus V}}\\
&=&\exp\bb{\sumit k 1 n \log\bb{1-\underset{x\in\partial D_k}\sup\;\omega\bb{x,V;D_{k+1}\setminus V}}}\\
&\ge&\exp\bb{-\alpha\sumit k 1 n \underset{x\in\partial D_k}\sup\;\omega\bb{x,V;D_{k+1}\setminus V}},
\end{eqnarray*}
since $\log(1-\eta)\ge-\eta\cdot \alpha$ for $\eta\in\sbb{0,1-\frac 1\alpha}$.\\
A similar computation shows that:
\begin{eqnarray*}
\omega\bb{0,\partial D_{n+1} ;D_{n+1} \setminus V}&\le&\cdots\le\exp\bb{-\sumit k 1 n \underset{x\in\partial D_k}\inf\;\omega\bb{x,V;D_{k+1}\setminus V}}.
\end{eqnarray*}
\end{proof}

\addcontentsline{toc}{section}{References}
\nocite{*}
\bibliographystyle{plain}
\bibliography{References}

\bigskip
\noindent A.G.:
Mathematical and Computational Sciences, University of Toronto Mississauga, Canada.
\newline{\tt adiglucksam@gmail.com} 

\end{document}